\documentclass[12pt,reqno]{amsart}
\usepackage{latexsym, amsmath, amssymb, amsthm, stmaryrd, bm} 
\usepackage[T1]{fontenc}
\usepackage{mathptmx}
\usepackage[colorlinks=true, linkcolor=blue, citecolor=blue]{hyperref}
\usepackage[shortalphabetic]{amsrefs}

\usepackage[centering, includeheadfoot, includehead, hmargin=1.1in,
  vmargin=1in]{geometry}
\newtheorem{lemma}{Lemma}
\newtheorem{theorem}[lemma]{Theorem}
\newtheorem{proposition}[lemma]{Proposition}
\theoremstyle{definition}
\newtheorem{conjecture}[lemma]{Conjecture}

\numberwithin{equation}{section}
\renewcommand{\geq}{\geqslant}
\renewcommand{\leq}{\leqslant}
\renewcommand{\AA}{\ensuremath{\mathbb{A}}} 
\newcommand{\NN}{\ensuremath{\mathbb{N}}} 
\newcommand{\FF}{\ensuremath{\mathbb{F}}} 
\newcommand{\CC}{\ensuremath{\mathbb{C}}} 
\newcommand{\PP}{\ensuremath{\mathbb{P}}} 
\newcommand{\QQ}{\ensuremath{\mathbb{Q}}} 
\newcommand{\RR}{\ensuremath{\mathbb{R}}} 
\newcommand{\ZZ}{\ensuremath{\mathbb{Z}}} 
\DeclareMathOperator{\variety}{V}
\newcommand{\cst}[1]{{\llbracket #1 \rrbracket}}
\newcommand{\eulerian}[2]{\genfrac{\langle}{\rangle}{0pt}{}{#1}{#2}}
\makeatletter
\renewcommand\subsection{\@startsection{subsection}{2}%
  \z@{.5\linespacing\@plus.7\linespacing}{-.5em}%
  {\normalfont\scshape}}
\makeatother

\begin{document}

\title[Laurent Polynomials and Eulerian Numbers]{Laurent
  Polynomials and Eulerian Numbers}

\author[D.~Erman]{Daniel Erman}
\address{\texttt{derman@math.berkeley.edu} : Department of
  Mathematics\\ University of California\\ Berkeley\\ CA \\
  94720-3840\\ USA}

\author[G.G.~Smith]{Gregory G. Smith}
\address{\texttt{ggsmith@mast.queensu.ca} : Department of Mathematics
  \& Statistics \\ Queen's University \\ Kingston \\ ON \\ K7L 3N6\\
  Canada}

\author[A.~V\'arilly-Alvarado]{Anthony V\'arilly-Alvarado}
\address{\texttt{varilly@rice.edu} : Department of Mathematics\\ MS
  136\\ Rice University\\ 6100 South Main Street\\ Houston\\ TX\\
  77005-1892\\ USA}



\begin{abstract}
  Duistermaat and van der Kallen show that there is no nontrivial
  complex Laurent polynomial all of whose powers have a zero constant
  term.  Inspired by this, Sturmfels poses two questions: Do the
  constant terms of a generic Laurent polynomial form a regular
  sequence?  If so, then what is the degree of the associated
  zero-dimensional ideal?  In this note, we prove that the Eulerian
  numbers provide the answer to the second question.  The proof
  involves reinterpreting the problem in terms of toric geometry.
\end{abstract}

\maketitle

\section{Motivation and Statement of Theorem}
\label{s:intro}

\noindent
In \cite{DvdK}, J.J.~Duistermaat and W. van der Kallen establish that,
for any Laurent polynomial $f \in \CC[z,z^{-1}]$ that is neither a
polynomial in $z$ nor $z^{-1}$, there exists a positive power of $f$
that has a nonzero constant term.  Motivated by this result,
Sturmfels~\cite{Sturmfels}*{\S2.5} asks for an effective version: Can
we enumerate the Laurent polynomials that have the longest possible
sequence of powers with zero constant terms?

By rephrasing this question in the language of commutative algebra,
Sturmfels also offers a two-step approach for answering it.
Specifically, consider the Laurent polynomial
\begin{equation}
  \tag{$\spadesuit$}
  \label{e:affine}
  f(z) := z^{-m} + x_{-m+1}^{\,} \, z^{-m+1} + \dotsb + x_{n-1}^{\,}
  \, z^{\, n-1} + z^{\, n}
\end{equation}
and, for any positive integer $i$, let $\cst{f^i}$ denote the constant
coefficient of the $i$-th power of $f$.  First, Problem~2.11 in
\cite{Sturmfels}*{\S2.5}, together with computational evidence,
suggests the following:

\begin{conjecture}
  \label{c:sturmfels}
  The coefficients
  $\cst{f^1}, \cst{f^2}, \dotsc, \cst{f^{m+n}}$ generate the unit
  ideal in the polynomial ring $\CC[x_{-m+1}, \dotsc, x_{n-1}]$.
\end{conjecture}

\noindent
Second, assuming this conjecture, Exercise~13 in
\cite{Sturmfels}*{\S2.6} asks for the degree of the ideal $I_{m,n} :=
\bigl\langle \cst{f^1}, \cst{f^2}, \dotsc, \cst{f^{m+n-1}}
\bigr\rangle$.  The zeros of $I_{m,n}$ would be the Laurent
polynomials of the form~\eqref{e:affine} that have the longest
possible sequence of powers with vanishing constant terms.

The goal of this article is to complete the second part.
Theorem~\ref{t:main} provides the unexpected and attractively simple
answer.  Following \cite{GKP}*{\S6.2}, the Eulerian number
$\eulerian{n}{k}$ is the number of permutations of $\{1, \dotsc, n\}$
with exactly $k$ ascents.

\begin{theorem}
  \label{t:main}
  If Conjecture~\ref{c:sturmfels} holds, then the degree of the ideal
  $I_{m,n}$ is $\eulerian{m+n-1}{m-1}$.
\end{theorem}

\noindent
This result is equivalent to saying that the dimension of
$\CC[x_{-m+1}, \dotsc,x_{n-1}]/I_{m,n}$, as a $\CC$\nobreakdash-vector
space, is $\eulerian{m+n-1}{m-1}$.

Notably, Theorem~\ref{t:main} gives a new interpretation for the
Eulerian numbers: $\eulerian{m+n-1}{m-1}$ enumerates certain Laurent
polynomials.  Even without Conjecture~\ref{c:sturmfels}, we show that
these Eulerian numbers count the solutions to certain systems of
polynomial equations; see Proposition~\ref{p:count}.  Despite
superficial similarities between our work and other appearances of
Eulerian numbers in algebraic geometry (e.g.{} \cites{BeckStapledon,
  Brenti, BrentiWelker, Hirzebruch, Stanley, Stembridge}), we know of
no substantive connection.

Our proof of Theorem~\ref{t:main}, given in \S\ref{s:proof}, recasts
the problem in terms of toric geometry --- we construe the degree of
$I_{m,n}$ as an intersection number on a toric compactification of the
space of Laurent polynomials of the form~\eqref{e:affine}.  Building
on this idea, \S\ref{s:refine} provides a recursive formula for the
degree of ideals similar to $I_{m,n}$ that arise from sparse Laurent
polynomials.  As a by-product, we give a geometric explanation for a
formula expressing $\eulerian{m+n-1}{m-1}$ as a sum of nonnegative
integers; see \eqref{e:decomp}.  We list several questions arising
from our work in \S\ref{s:open}.

\subsection*{Acknowledgements}

We thank David Eisenbud, Alexander Postnikov, Bernd Sturmfels, and
Mauricio Velasco for useful conversations.  The computer software
\emph{Macaulay~2}~\cite{M2} was indispensable in discovering both the
statement and proof of the theorem.  The Mathematical Sciences
Research Institute (MSRI), in Berkeley, provided congenial
surroundings in which much of this work was done.  Erman was partially
supported by a NDSEG fellowship, Smith was partially supported by
NSERC, and V\'arilly-Alvarado was partially supported by a Marie Curie
Research Training Network within the Sixth European Community
Framework Program.

\section{Toric Reinterpretation}
\label{s:proof}

\noindent
This section proves Theorem~\ref{t:main} by reinterpreting the degree
of $I_{m,n}$ as an intersection number on a projective variety
$X(m,n)$.  Subsection \S\ref{s:homogeneous} introduces a
homogenization of the ideal $I_{m,n}$, \S\ref{s:toric} describes the
toric variety $X(m,n)$, and \S\ref{s:chow} computes the required
intersection number.

\subsection{Homogenization}
\label{s:homogeneous}

For positive integers $m$ and $n$, consider the Laurent polynomial
\[
\tilde{f} :=x_{-m}^{\,} \, z^{-m} + x_{-m+1}^{\,} \, z^{-m+1} +
\dotsb + x_{n-1}^{\,}\, z^{\, n-1} + x_n^{\,} \, z^{\, n} \, ,
\]
and, for any positive integer $i$, let $\cst{\tilde{f}^i}$ denote the
constant coefficient of the $i$-th power of $\tilde{f}$.  Let $S$ be
the polynomial ring $\CC[x_{-m}, \dotsc, x_{n}]$ and let $J$ be the
$S$-ideal $\bigl\langle\cst{\tilde{f}^1}, \cst{\tilde{f}^2}, \dotsc,
\cst{\tilde{f}^{m+n-1}} \bigr\rangle$.  The $\CC$\nobreakdash-valued
points of $\variety(J) \subset \AA^{m+n+1}$ are precisely the Laurent
polynomials for which the constant term of the first $m+n-1$ powers
vanishes. Since $J$ is contained in the reduced monomial ideal $B :=
\langle x_{-m}, \dotsc, x_{-1} \rangle \cap \langle x_{0}, x_{1},
\dotsc, x_{n} \rangle$, the $\CC$\nobreakdash-valued points of
$\variety(J)$ not contained in $\variety(B)$ give rise to Laurent
polynomials that are neither polynomials in $z$ nor $z^{-1}$.

To understand the ideal $J$ more explicitly, let $\bm{w} := [
\begin{smallmatrix}
  -m & \dotsb & n
\end{smallmatrix}]^{\textsf{t}} \in \ZZ^{m+n+1}$.  If $\bm{u} \in
\NN^{m+n+1}$, then the multinomial theorem \cite{GKP}*{p.{} 168}
implies that
\[
\cst{\tilde{f}^i} = \sum_{
  \begin{subarray}{c}
    |\bm{u}| = i \\
    \bm{w} \cdot \bm{u} = 0
  \end{subarray}}
\binom{i}{\bm{u}} \bm{x}^{\bm{u}} = \sum_{
  \begin{subarray}{c}
    |\bm{u}| = i \\
    \bm{w} \cdot \bm{u} = 0
  \end{subarray}} \binom{i}{u_{1}, \dotsc, u_{m+n+1}} x_{-m}^{u_{1}}
\, x_{-m+1}^{u_{2}} \dotsb \, x_{n}^{u_{m+n+1}} \, .
\]
Hence, for all positive integers $i$, the polynomial
$\cst{\tilde{f}^i}$ is homogeneous of degree $\left[ 
  \begin{smallmatrix} 
    i \\ 0 
  \end{smallmatrix} \right]$ with respect to the
$\ZZ^2$\nobreakdash-grading of $S$ induced by
setting $\deg(x_{j}) := \left[ \begin{smallmatrix} 1 \\
    j \end{smallmatrix} \right] \in \ZZ^2$ for all $-m \leq j \leq n$.
In particular, the ideal $J$ is invariant under the automorphism of
$S$ determined by the map $\tilde{f}(z) \mapsto \lambda \tilde{f}(\xi
z)$ where $\lambda, \xi \in \CC^*$.  Moreover, if $x_{-m}$ and $x_n$
are both nonzero, then there exist scalars $\lambda$, $\xi \in \CC^*$
such that the image of $\tilde{f}$ under this $(\CC^*)^2$-action has
the form \eqref{e:affine}.

\subsection{Toric Variety}
\label{s:toric}

When $m + n > 2$, let $X(m,n)$ be the toric variety with total
coordinate ring $S$ (a.k.a.{} the Cox ring) and irrelevant ideal $B$;
see \cite{Cox}*{\S2}.  The variety $X(m,n)$ provides a toric
compactification for the space of all Laurent polynomials of the form
\eqref{e:affine}.  When no confusion is likely, we simply write $X$ in
place of $X(m,n)$.  Proposition~2.4 in \cite{Cox} shows that
homogeneous $S$-ideals (up to $B$-torsion) correspond to closed
subschemes of $X$.  Hence, the ideal $J$ determines a closed subscheme
$\variety_{\! X}(J)$ of $X$.  If $x_{-m}x_n$ is a nonzerodivisor on
$\variety_{\! X}(J)$, then \S\ref{s:homogeneous} shows that the degree
of the ideal $I_{m,n}$ equals the degree of $\variety_{\! X}(J)$.  We
prove Theorem~\ref{t:main} by computing the latter degree.

More concretely, $X$ is the toric variety associated to the following
strongly convex rational polyhedral fan $\Sigma$; see
\cite{Fulton}*{\S1.4}.  The lattice of one-parameter subgroups is $N =
\ZZ^{m+n-1}$ and the rays (i.e.{} one-dimensional cones) in the fan
$\Sigma$ are generated by the columns of the matrix:
\begin{align}
  \tag{$\clubsuit$}
  \label{e:matrix}
  \left[
    \begin{matrix}
      1 & -2 & 1 & 0 & \dotsb & 0 \\
      2 & -3 & 0 & 1 & \dotsb & 0 \\
      \vdots & \;\;\;\vdots & \vdots & \vdots & \ddots & \vdots \\
     m+n-1 & -m-n & 0 & 0 & \dotsb & 1
    \end{matrix}
  \right] \, .
\end{align}
With the column ordering, we label the rays in $\Sigma$ by $\rho_{-m},
\dotsc, \rho_n$.  For integers $1 \leq i \leq m$ and $0 \leq j \leq
n$, let $\sigma_{i,j}$ be the cone in $\RR^{m+n-1} = N \otimes_{\ZZ}
\RR$ spanned by all the rays except $\rho_{-i}$ and $\rho_j$.  The fan
$\Sigma$ is defined by taking these $\sigma_{i,j}$ as the maximal
cones.  By construction, $X$ is a singular simplicial projective toric
variety of dimension $m+n-1$.

\subsection{Intersection Theory}
\label{s:chow}

Since $X$ is a simplicial toric variety, its rational Chow ring
$A^*(X)_\QQ$ has an explicit presentation; see \cite{Fulton}*{\S5.2}.
Specifically, if $D_j$ is the torus-invariant Weil divisor associated
to the ray $\rho_j$ for all $-m \leq j \leq n$, then we have
\[
A^*(X)_\QQ = \frac{\QQ[D_{-m}, \dotsc, D_n]}{M+L}
\] 
where the monomial ideal $M := \langle D_{-m} D_{-m+1} \dotsb
D_{-1},D_0D_{1} \dotsb D_{n-1}D_{n} \rangle$ is the Alexander dual of
$B$, and the linear ideal $L := \langle iD_{-m} - (i+1)D_{-m+1} +
D_{-m+i+1} : 1 \leq i \leq m+n-1 \rangle$ encodes the rows of the
matrix \eqref{e:matrix}.

Choosing a shelling for the fan $\Sigma$ yields a distinguished basis
for $A^*(X)_\QQ$; again see \cite{Fulton}*{\S5.2}.  With this in mind,
we order the maximal cones of $\Sigma$ by $\sigma_{i,j} >
\sigma_{k,\ell}$ if $i+j > k+\ell$ or $i+j = k+\ell$ and $j > \ell$.
Let $\tau_{i,j}$ be the subcone of $\sigma_{i,j}$ obtained by
intersecting the maximal cone $\sigma_{i,j}$ with all cones
$\sigma_{k,\ell}$ satisfying $\sigma_{k,\ell} > \sigma_{i,j}$ and
$\dim \sigma_{i,j} \cap \sigma_{k,\ell} = m+n-2$.  We obtain a
shelling for $\Sigma$ (i.e.{} condition $(*)$ in \cite{Fulton}*{p.{}
  101} is satisfied) because $\dim \sigma_{i,j} \cap \sigma_{k,\ell} =
m+n-2$ if and only if $i = k$ and $j \neq \ell$ or $i \neq k$ and $j =
\ell$, so $\tau_{i,j} = \sigma_{i,j} \cap \bigl( \bigcap_{k > i}
\sigma_{k,j} \bigr) \cap \bigl( \bigcap_{\ell > j} \sigma_{i,\ell}
\bigr)$.  Hence, the collection $\{ [\variety(\tau_{i,j})] \}$ forms a
basis for $A^*(X)_\QQ$.

Set $D_{(-i,j)} := D_{-i+1} \dotsb D_{-1} \cdot D_0 \dotsb D_{j-1}$;
the empty product $D_{(-1,0)} = 1$ is the unit in $A^*(X)_\QQ$.  The
generators of $M$ imply that $D_{(-i,j)} = 0$ in $A^*(X)_\QQ$ if $i >
m$ or $j > n$. Since $\tau_{i,j}$ is spanned by the rays $\rho_\ell$
with $-i < \ell < j$, it follows that $[\variety(\tau_{i,j})] =
D_{(-i,j)}$.  Thus, $D_{(-i,j)}$ for $1 \leq i \leq m$ and $0 \leq j
\leq n$ forms a basis for $A^*(X)_\QQ$.  The degree of a
zero-dimensional subscheme $Y$ of $X$, denoted $\deg(Y)$, is the
rational number such that $[Y] = \deg(Y) \, D_{(-m,n)}$ in
$A^{m+n-1}(X)_{\QQ}$.

The following calculation is the key to proving Theorem~\ref{t:main}.

\begin{lemma}
  \label{l:chow}
  For $1 \leq k \leq m+n-1$, we have $k! \, D_{0}^{\, k} =
  \sum\limits_{i=1}^{k} \eulerian{k}{i-1} D_{(-i,k-i+1)}$ in
  $A^*(X)_\QQ$.
\end{lemma}

\begin{proof}
  In the polynomial ring $\QQ[z]$, Worpitzky's identity is $z^k =
  \sum_i \eulerian{k}{i} \binom{z+i}{k}$; see eq. (6.37) in
  \cite{GKP}*{p.{} 255} or for a combinatorial proof see
  \cite{BG}*{\S7}.  Rearranging, and homogenizing this identity gives
  the equation $k! \, z^k = \sum_{i=1}^{k} \eulerian{k}{i-1} \bigl( z
  + (i-1)y \bigr) \bigl( z + (i-2)y \bigr) \dotsb \bigl( z + (i-k)y
  \bigr)$, in the $\ZZ$-graded polynomial ring $\QQ[z,y]$ with
  $\deg(z) = \deg(y) = 1$.  Under the substitution $z \mapsto D_{0}$
  and $y \mapsto D_{1} - D_{0}$, we obtain the equation
  \[
  k! \, D_{0}^k = \sum_{i=1}^{k} \eulerian{k}{i-1} \Bigl(
  \bigl(1-(i-1) \bigr)D_{0} - (i-1)D_1 \Bigr) \dotsb \Bigl( \!
  \bigl(1-(i-k) \bigr)D_{0} - (i-k)D_{1} \! \Bigr)
  \]
  in $A^*(X)_\QQ$.  To complete the proof, we observe that the ideal
  $L$ contains the linear relation $D_{i} = (1-i)D_{0} - iD_{1}$ for
  all $-m \leq i \leq n$.
\end{proof}

Using this lemma, we can compute the degree of certain complete
intersections in $X$.

\begin{proposition}
  \label{p:count}
  Let $g_1, \dotsc, g_{m+n-1}$ be homogeneous elements of $S$ such
  that $\deg(g_j) = \left[ \begin{smallmatrix} j \\
      0 \end{smallmatrix} \right]$ for $1 \leq j \leq m+n-1$.  If $\;
  \variety_{\! X}(g_1,\dotsc,g_{m+n-1})$ is a zero-dimensional
  subscheme of $X$, then its degree is $\eulerian{m+n-1}{m-1}$.
\end{proposition}

\begin{proof}
  Each homogeneous polynomial $g_j$ defines a hypersurface in $X$.
  This Cartier divisor is rationally equivalent to $j \, D_{0}$
  because we have $\deg(g_j) = \left[
    \begin{smallmatrix} j \\ 0 \end{smallmatrix} \right]$ for $1 \leq
  j \leq m+n-1$.  The subscheme $Z := \variety_{\!
    X}(g_1,\dotsc,g_{m+n-1})$ has dimension zero if and only if it is
  a complete intersection.  Hence, the degree of $Z$ equals the
  appropriate intersection number, namely the coefficient of
  $D_{(-m,n)}$ in $\prod_{j=1}^{m+n-1} j \, D_{0}$; see
  Proposition~7.1 in \cite{Fulton2}.  Since $D_{(-i,k-i+1)}= 0$ for $i
  > m$ or $k-i+1 > n$, Lemma~\ref{l:chow} yields $\prod_{j=1}^{m+n-1}
  j \, D_{0} = (m+n-1)! \, D_{0}^{\, m+n-1} = \eulerian{m+n-1}{m-1}
  D_{(-m,n),}$.
\end{proof}

\begin{proof}[Proof of Theorem~\ref{t:main}]
  Applying Conjecture~\ref{c:sturmfels} for the pairs of positive
  integers $(m,n-1)$ and $(m-1,n)$, we see that $\variety_{\! X}(J)
  \cap \variety_{\! X}(x_{-m}x_n) = \varnothing$.  It follows that
  $[\variety_{\! X}(J)]$ belongs to the socle of $A^*(X)_{\QQ}$ and
  thus $\variety_{\! X}(J)$ has dimension zero.  Since $x_{-m}x_n$ is
  a nonzerodivisor on $\variety_{\! X}(J)$, we see that
  $\deg(I_{m,n})$ equals $\deg \variety_{\! X}(J)$; see
  \S\ref{s:toric}.  Therefore, applying Proposition~\ref{p:count}
  completes the proof.
\end{proof}

\section{Sparse Laurent Polynomials}
\label{s:refine}

\noindent
In this section, we compute the degree of subschemes of $X(m,n)$
corresponding to certain sparse Laurent polynomials. Given the
recurrence relation that these degrees satisfy, they may be regarded
as a generalized form of Eulerian numbers. This computation also
generates a decomposition of $\eulerian{m+n-1}{m-1}$ as a sum of
nonnegative integers; see \eqref{e:decomp}.

Fix a pair of positive integers $(m,n)$ and let $d$ be a positive
integer dividing $m+n$.  Consider the closed subscheme $X_d$ of $X$
corresponding to Laurent polynomials of the form
\[
x_{-m} \, z^{-m} + x_{-m+d}^{\,} \, z^{-m+d} + \dotsb + x_{n-d}^{\,} \, z^{\,
  n-d} + x_{n} \, z^{\, n} \, .
\]
In other words, $X_d$ is the subscheme of $X$ defined by the monomial
ideal generated by the variables not belonging to $\{x_{-m},x_{-m+d},
\dotsc, x_{n-d}, x_{n} \}$.  When $d = 1$, we have $X_d = X$.

For $1 \leq j \leq m+n-1$, let $g_j$ be a generic polynomial in $S$ of
degree $\left[ \begin{smallmatrix} j \\ 0 \end{smallmatrix} \right]$.
These generic polynomials cut out the subscheme $Z := \variety_{\!
  X}(g_1, \dotsc, g_{m+n-1})$.  Consider $Z_d := Z \cap X_d$.  To
compute the degree of $Z_d$, we introduce the following notation.  If
$0 \leq \ell \leq d-1$, then we define
\[
\eulerian{d-1}{\ell}_{\!\! d} :=
\begin{cases}
  0 & \text{if $\gcd(\ell+1,d) \neq 1$,} \\
  1 & \text{if $\gcd(\ell+1,d) = 1$,}
\end{cases}
\]
and we extend the definition of $\eulerian{k}{\ell}_{\! d}$ for all
triples $(k,\ell,d)$ such that $d$ divides $k+1$ via
\[
\eulerian{k}{\ell}_{\!\! d} := (\ell+1)
\eulerian{k-d}{\ell}_{\!\! d} + (k-\ell)
\eulerian{k-d}{\ell-d}_{\!\! d} \, .
\]
It follows that $\eulerian{k}{\ell} = \eulerian{k}{\ell}_{\! 1}$.

\begin{proposition}
  \label{p:degZ}
  The scheme $Z_d$ has dimension zero and degree
  $\eulerian{m+n-1}{m-1}_{\! d}$ when $\gcd(d,n) = 1$; otherwise
  the scheme $Z_d$ is empty.
\end{proposition}

Before proving this proposition, we record a technical lemma. Let
$W_i$ be the vector space of all polynomials in $S$ of degree $\left[
  \begin{smallmatrix} i \\ 0 \end{smallmatrix} \right]$ with support
contained in $\{x_{-m}, x_{-m+d}, \dotsc, x_{n-d}, x_{n} \}$. Given a
subset $\mathcal{S} \subseteq \{d,2d,\dotsc,m+n-d\}$, let
$D(\mathcal{S})$ be the subscheme of $X_d$ defined by the ideal
generated by $W_i$ for all $i \in \mathcal{S}$.

\begin{lemma}
  \label{l:poly}
  If $\mathcal{S} \subseteq \{ d, 2d, \dotsc, m+n-d \}$, then $\dim
  D(\mathcal{S}) \leq \frac{m+n}{d} -1 -|\mathcal{S}|$.
\end{lemma}

\begin{proof}
  It suffices to show that $D(\mathcal{S})$ is contained in a finite
  union of subschemes with dimension $\frac{m+n}{d} - 1 -
  |\mathcal{S}|$.  To a point $P = [p_{-m} : p_{-m+d} : \dotsc : p_n]$
  in the subscheme $D(\mathcal{S})$, we associate the support sets
  $\mathcal{E}_+ := \{ i \geq 0 \mid p_i \neq 0 \}$ and
  $\mathcal{E}_- := \{ i>0 \mid p_{-i} \neq 0 \}$. From the definition
  of $X_d$, we deduce that
  $\mathcal{E}_{+} \subseteq \{m,m-d,\dotsc\}$ and $\mathcal{E}_-
  \subseteq \{n,n-d,\dotsc\}$.  Observe that $P$ lies in the subspace
  defined by the ideal $\langle x_i \mid i \in \{-m,-m+d, \dots, n\}
  \setminus (\mathcal{E}_+ \cup \mathcal{E}_-) \} \rangle$ and that
  this subspace has dimension $|\mathcal{E}_+| + |\mathcal{E}_-|-2$.
  Hence, it is enough to prove $|\mathcal{E}_+| + |\mathcal{E}_-| - 2
  \leq \frac{m+n}{d} -1 -|\mathcal{S}| = |\mathcal{S}^\complement|$
  where $\mathcal{S}^\complement := \{d, 2d, \dotsc, m+n-d \}
  \setminus \mathcal{S}$.  To accomplish this, we consider the set
  \[
  \mathcal{P} := \{ i + j \mid \text{$i \in \mathcal{E}_+$, $j \in
    \mathcal{E}_-$, and $i+j \leq m+n-d$} \} \subseteq \{ d, 2d,
  \dotsc, m+n-d \} \, .
  \]
  To conclude, one verifies that $\mathcal{P} \subseteq
  \mathcal{S}^\complement$ and that $|\mathcal{E}_+| + |\mathcal{E}_-|
  -2 \leq |\mathcal{P}|$.
\end{proof}

\begin{proof}[Sketch of the Proof for Proposition~\ref{p:degZ}]
  To begin, we assume that $\gcd(d,n) = 1$. Let $\PP(W)$ be the
  product $\PP(W_d) \times \PP(W_{2d}) \times \dotsb \times
  \PP(W_{m+n-d})$ and consider the incidence variety
  \[
  U := \bigl\{ \bigl(P,(h_d,\dotsc,h_{m+n-d}) \bigr) \mid h_d(P) =
  \dotsb = h_{m+n-d}(P) = 0 \bigr\} \subseteq X_d \times \PP(W)
  \]
  with projection maps $\pi_1 \colon U \to X_d$ and $\pi_2 \colon U
  \to \PP(W)$. We claim that $\dim U \leq \dim \PP(W)$. To see this,
  observe that a general point $Q$ in $X_d$ does not belong to the
  base locus of any $W_i$, so the fiber $\pi_1^{-1}(Q)$ has dimension
  $\dim \PP(W) - \frac{m+n}{d} +1$.  One must also consider the
  dimensions of the various $\pi_1^{-1}\bigl( D(\mathcal{S}) \bigr)$,
  but Lemma~\ref{l:poly} shows that none of these preimages has
  dimension greater than $\dim \PP(W)$. Since $Z_d$ equals the fiber
  of $\pi_2$ over a general point of $\PP(W)$, the inequality $\dim U
  \leq \dim \PP(W)$ implies that $Z_d$ has dimension zero.  The
  appropriate modifications to the proofs of Lemma~\ref{l:chow} and
  Proposition~\ref{p:count} show that the degree of $Z_d$ is
  $\eulerian{m+n-1}{m-1}_{\! d}$.

  Assume that $e := \gcd(d,n) > 1$. If $m' := m/e$, $n' := n/e$, and
  $d' := d/e$, then there is an isomorphism $X_d = X(m,n)_d
  \xrightarrow{\cong} X(m',n')_{d'} = X'_{d'}$. Under this
  identification, $Z_d$ is determined by the ideal $\langle g_{d'},
  g_{2d'}, \dotsc, g_{e(n'+m')-d'} \rangle$.  Let $U'$ be the
  incidence variety for the parameters $(m',n',d')$.  From the proof
  of Lemma~\ref{l:poly}, we deduce that $x_{-m'}x_{n'}$ is a
  nonzerodivisor on the top dimensional components of $U'$.  Hence,
  the generic polynomial $g_{m'+n'}$ is also a nonzerodivisor on $U'$,
  so the intersection of the general fibre of $\pi'_2 \colon U' \to
  \PP(W)$ with the hypersurface defined by $g_{m'+n'}$ is
  empty. Therefore, we have $Z_d = \varnothing$.
\end{proof}

To obtain a decomposition for the Eulerian numbers, we stratify the
generic complete intersection $Z$ by singularity type.  Let
$X_d^\circ$ be the open subscheme of $X_d$ consisting of all
singularities of type $B(\ZZ/d\ZZ)$ in $X$.  Each point in $Z$ belongs
to $X_d$ for some $d$ that divides $m+n$.  Setting $Z_d^\circ := Z
\cap X_d^{\circ}$, we obtain
\begin{equation}
  \tag{$\heartsuit$}
  \label{e:decomp}
  \eulerian{m+n-1}{m-1} = \deg(Z) =  \sum_{d \mid
    m+n} \deg(Z_d^{\circ}).
\end{equation}
Moreover, M\"obius inversion and Proposition~\ref{p:degZ} yield
\[
\deg(Z_d^{\circ}) = \sum\nolimits_{c \mid (m+n)/d} \mu(c)
\eulerian{m+n-1}{m-1}_{\!\! cd} \, ,
\]  
where $\mu$ is the classical M\"obius function; see eq. (4.55) and
(4.56) in \cite{GKP}*{p.{} 136}.

The equation \eqref{e:decomp} has an elegant combinatorial refinement
which we learnt from Alexander Postnikov; cf.{} \cite{LP}*{\S6}.  To
sketch this refinement, we observe that the Eulerian number
$\eulerian{n}{k}$ also counts the circular permutations of
$\{0,\dotsc,n\}$ with $k+1$ circular ascents.  The group
$\ZZ/(n+1)\ZZ$ naturally acts on this subset of circular permutations;
add $1$ modulo $n+1$ to each element.  The cardinalities of the orbits
then give rise to \eqref{e:decomp}.  More precisely,
$\deg(Z_d^{\circ})$ equals the product of $(m+n)/d$ and the number of
orbits with cardinality $(m+n)/d$.  For example, if $m = 2$ and $n =
3$, then we have $\eulerian{4}{1} = 11$, $\deg(Z_5^{\circ}) = 1$, and
$\deg(Z_1^{\circ}) = 10 = 2 \cdot 5$.  On the other hand, the eleven
circular permutations of $\{0, \dotsc, 4\}$ with two circular ascents
are partitioned into three $\ZZ/5\ZZ$-orbits, namely
$\{0\,3\,2\,4\,1\}$, $\{ 0\,1\,4\,3\,2, 0\,4\,3\,1\,2, 0\,4\,2\,3\,1,
0\,3\,4\,2\,1, 0\,3\,2\,1\,4\}$, and $\{ 0\,2\,1\,4\,3, 0\,4\,1\,3\,2,
0\,2\,4\,3\,1, 0\,4\,2\,1\,3, 0\,3\,2\,4\,1\}$.

\section{Further Questions}
\label{s:open}

\subsection{Regular sequence}
\label{s:regseq}

Theorem~\ref{t:main} underscores the significance of
Conjecture~\ref{c:sturmfels}.  To prove this conjecture, it would be
enough to show that $\variety_{\! X}( \cst{\tilde{f}^1}, \dotsc,
\cst{\tilde{f}^{m+n}} )$ is the empty set.  From this perspective, the
proof of Proposition~\ref{p:degZ} could be viewed as evidence
supporting this conjecture: for generic elements $g_j$
of $S$ with degree $\left[ \begin{smallmatrix} j \\
    0 \end{smallmatrix} \right]$, the subscheme $\variety_{\! X}(g_1,
\dotsc, g_{m+n})$ is indeed empty.  On the other hand,
Conjecture~\ref{c:sturmfels} is false over a field with positive
characteristic. For instance, if $f := z^{-1} + z \in
\FF_2[z,z^{-1}]$, then we have $\cst{f^i} = 0$ for all $i$.  Even if
Conjecture~\ref{c:sturmfels} holds, the $\FF_p$\nobreakdash-vector
space $\FF_p[x_{-m+1}, \dotsc,x_{n-1}]/I_{m,n}$ may fail to have a
finite dimension; this happens when $p=2$, $m=1$, and $n=2$.

\subsection{Combinatorics}
\label{s:comb}

The positivity and simplicity of many formulae in this article suggest
that we have uncovered only part of the combinatorial structure.  To
help orient the search for further structure, we pose two specific
questions:
\begin{itemize}
\item Can one find an explicit basis for $\CC[x_{-m+1}, \dots,
  x_{n-1}]/I_{m,n}$ together with a bijection to the permutations of
  $[m+n-1]$ with exactly $m-1$ ascents?
\item Does $\displaystyle\sum_{j \geq 0} \dim_\CC \bigl(
  \tfrac{S}{\langle g_1, \dotsc, g_{m+n} \rangle}
  \bigr)_{\bigl[ \begin{smallmatrix} j \\ 0 \end{smallmatrix} \bigr]}
  = \textstyle\eulerian{m+n-1}{m-1} $ hold for all positive $m$ and
  $n$?  When $m = 3$ and $n = 3$, we have $\textstyle\eulerian{5}{2} =
  66 = 1 + 0 + 2 + 3 + 6 + 7 + 9 + 10 + 9 + 7 + 6 + 3 + 2 + 0 + 1$.
\end{itemize}

\raggedright
 
\end{document}